%% file: paper.tex
\pdfoutput=1
\documentclass[11pt,a4paper]{article}
\usepackage[T1]{fontenc}
\usepackage{amsmath,amssymb,amsthm}
\usepackage{lmodern}
\usepackage{newtxtext,newtxmath}
\usepackage[margin=28mm,headheight=14pt]{geometry}
\usepackage{microtype,booktabs,array,enumitem}
\usepackage[dvipsnames]{xcolor}
\usepackage{fancyhdr}
\usepackage[colorlinks=true,linkcolor=MidnightBlue,citecolor=MidnightBlue,urlcolor=MidnightBlue]{hyperref}
\usepackage[nameinlink,noabbrev,capitalise]{cleveref}
\hypersetup{pdftitle={A Padovan-automatic description of a nested recurrence},
  pdfauthor={Benoit Cloitre, Haobo Ma, and Wenlin Zhang},
  pdfsubject={Research note with appendices},
  pdfkeywords={nested recurrence, Padovan numeration, finite automata, morphic words}}
\numberwithin{equation}{section}
\newtheorem{theorem}{Theorem}[section]
\newtheorem{proposition}[theorem]{Proposition}
\newtheorem{lemma}[theorem]{Lemma}
\newtheorem{corollary}[theorem]{Corollary}
\theoremstyle{remark}

\newcommand{\Z}{\mathbb{Z}}
\newcommand{\eps}{\varepsilon}
\newcommand{\rep}{\operatorname{rep}}
\newcommand{\val}{\operatorname{val}}
\newcommand{\sh}{\operatorname{shift}_2}
\newcommand{\Add}{\mathsf{Add}}
\newcommand{\Graph}{\mathsf{B}}
\newcommand{\Good}{\mathsf{Good}}
\newcommand{\code}[1]{\texttt{#1}}
\setlist{nosep}
\title{A Padovan-automatic description\\of a nested recurrence}
\author{
Benoit Cloitre\\[-2pt]
\small ORCID \href{https://orcid.org/0009-0001-6778-153X}{\texttt{0009-0001-6778-153X}}
\and
Haobo Ma\\[-2pt]
\small ChronoAI Pte Ltd; The Omega Institute\\[-2pt]
\small\texttt{auric@aelf.io}
\and
Wenlin Zhang\\[-2pt]
\small National University of Singapore (NUS); The Omega Institute\\[-2pt]
\small\texttt{e1327962@u.nus.edu}
}
\date{September 25, 2026}

\begin{document}
\maketitle
\begin{abstract}
We study the sequence $a(0)=0$, $a(1)=1$ and
$a(n)=n-a(n-a(n-a(n-1)))$ for $n\ge2$, listed as A076502 in the
On-Line Encyclopedia of Integer Sequences. We identify $a(n)$ as a
two-position shift in the greedy Padovan numeration system, with a finite-state
correction. The proof constructs an addition automaton from an exact
integer-carry invariant and certifies its completeness by finite-language
inclusion; a synchronized automaton then verifies the nested recurrence.
We establish bounded discrepancy from the line of slope $c$, where
$c^3-c^2+2c-1=0$, and show that the exact set of offsets from $\lfloor cn\rfloor$
is $\{-1,0,1,2\}$. We construct an explicit $26$-letter non-erasing morphic
presentation of the first-difference word, prove that its least balance
constant is $4$, and give an effective procedure for enclosing the global
discrepancy extrema to arbitrary accuracy. We formalize the recurrence
identification, six-decimal discrepancy bound, exact offset set, concrete
morphic identity, least balance constant, and an effective extrema algorithm
in Lean. Separate exact computations refine the numerical enclosures.
\end{abstract}

\section{Introduction}
Consider the nested recurrence
\begin{equation}\label{eq:recurrence}
 a(0)=0,\qquad a(1)=1,\qquad
 a(n)=n-a\bigl(n-a\bigl(n-a(n-1)\bigr)\bigr)\quad(n\ge2).
\end{equation}
Its positive-index terms form sequence A076502 in the OEIS~\cite{OEIS}.
An original conjectural description placed $a(n)$ within the three offsets
$0,1,2$ of $\lfloor cn\rfloor$, where $c$ is the real root of
\begin{equation}\label{eq:cubic}
 P(t)=t^3-t^2+2t-1.
\end{equation}
The example $a(1167)=664$ and $\lfloor1167c\rfloor=665$ rules out that
offset set. We establish bounded discrepancy and the asymptotic relation
$a(n)\sim cn$, and prove that the exact offset set is $\{-1,0,1,2\}$.

Let $\rho>1$ be the plastic number, the real root of $\rho^3=\rho+1$
(OEIS A060006), and set $c=\rho^{-2}$.
Substitution shows $P(c)=0$; moreover $P'(t)=3t^2-2t+2>0$, so this specifies
the same constant as in~\eqref{eq:cubic}. Our numeration system uses
\begin{equation}\label{eq:weights}
 U_0,U_1,U_2,U_3,U_4=1,2,3,4,5,\qquad
 U_p=U_{p-2}+U_{p-3}\quad(p\ge4).
\end{equation}
These are the distinct positive Padovan numbers (OEIS A000931).
The recurrence $U_{n+4}=U_{n+2}+U_{n+1}$, valid for all $n\ge0$, has characteristic
polynomial $X(X^3-X-1)$: the root $0$ accounts for an initial transient.
Unlike the canonical $k=5$ polynomial of~\cite{Letouzey},
$X^5-X^4-1=(X^2-X+1)(X^3-X-1)$, it has no unit-circle roots.
Digits are read from most to least significant. Write $\rep(n)$ for the
canonical greedy representation of $n$, with $\rep(0)=\eps$. If its selected
positions are $J$, define
\begin{equation}\label{eq:shift}
 \sh(n)=\sum_{\substack{p\in J\\p\ge2}} U_{p-2}.
\end{equation}
Thus the shift discards the two least significant digits. A deterministic
finite automaton with output (DFAO) assigns an output to the state reached
after reading a word.

\begin{theorem}\label{thm:main}
The recurrence~\eqref{eq:recurrence} defines a unique sequence of nonnegative
integers. For the explicit $28$-state DFAOs $E,D$ in Appendix Table~\ref{tab:dfao},
\begin{align}
 a(n)&=\sh(n)+E(\rep(n)),\label{eq:identification}\\
 a(n+1)-a(n)&=D(\rep(n))\in\{0,1\}\label{eq:difference}
\end{align}
for every $n\ge0$. Its discrepancy satisfies
\begin{equation}\label{eq:mainbound}
 -1.049234<a(n)-cn<1.155081\qquad(n\ge0).
\end{equation}
Consequently $a(n)/n\to c$, and
\begin{equation}\label{eq:offsetset}
 \{a(n)-\lfloor cn\rfloor:n\ge1\}=\{-1,0,1,2\}.
\end{equation}
The difference word is morphic, not ultimately periodic, and has least
balance constant $4$. It admits the
explicit $26$-letter non-erasing presentation in \cref{tab:morphism}.
The global infimum and supremum of $a(n)-cn$ are effectively approximable
by rational intervals with a certified error tending to zero.
\end{theorem}

The central step is \eqref{eq:identification}. We build an addition
automaton for the Padovan system, prove it sound by a carry invariant
and complete by a finite inclusion of languages, and verify the
recurrence on synchronized representations. Strong induction identifies the
resulting function with~\eqref{eq:recurrence}. The remaining conclusions
follow from a contracting quadratic form, finite weighted-path optimization,
and the representation tree.

Nested recurrences with morphic interpretations have been studied
systematically by Celaya and Ruskey~\cite{CR}. Shallit's treatment of the
Narayana system~\cite{Shallit} is a close methodological precedent: it uses
synchronized automata and first-order verification to establish identities
for related nested recurrences. Our proof uses this automata method with a
directly certified adder for the particular weights~\eqref{eq:weights}.
The passage from abstract-numeration automaticity to morphicity is due to
Rigo and Maes~\cite{RM}; we check its representation-order hypothesis
explicitly and also provide the substitution itself. Our contributions are
the identification of A076502, its discrepancy bounds, and an explicit
morphic presentation.

Recent work of Letouzey, Li, and Steiner~\cite{LLS}, and of
Letouzey~\cite{Letouzey}, develops the canonical-numeration, shift, and
discrepancy approach for generalized Hofstadter functions. The latter work
also supplies formalized finite optimization and tail bounds for global
discrepancy extrema. \Cref{sec:discrepancy} follows their finite-tail method.
Their iterated-composition recurrence differs from
\eqref{eq:recurrence}, so identifying A076502 and handling its exceptional
initial weight require separate arguments. We specialize the method to
the correction automaton $E$ and use a real quadratic form in
place of explicit complex-root estimates.

Finite-state normalization is classical in Pisot bases~\cite{Frougny}
and in linear numeration systems whose characteristic polynomial is
the minimal polynomial of a Pisot number~\cite{FS96}; see also~\cite{FS}
on finite beta-expansions and~\cite{Hollander} on regular greedy languages.
There is an initial-value
issue here: the minimal cubic recurrence fails at $p=3$, since $4\ne2+1$.
Indeed, a strictly increasing integer basis starting with $U_0=1$ cannot
satisfy $U_3=U_1+U_0$, since no integer $U_2$ lies strictly between $U_1$
and $U_1+1$. The extended annihilator $X(X^5-X^4-1)$ follows by
multiplication by $X^2-X+1$, but does not remove this exception. We prove
greedy regularity and addition directly for these initial values, using a
final-digit carry correction to handle the exceptional weight.

\section{Greedy representations and certified addition}\label{sec:arithmetic}
For a binary word $w=w_k\cdots w_0$, let
$\val(w)=\sum_{p=0}^k w_pU_p$. Let $G$ denote the language of greedy
representations with arbitrary leading zeros, including the empty word.

\begin{lemma}\label[lemma]{lem:greedy}
A binary word belongs to $G$ if and only if consecutive positions carrying
$1$ are at distance at least five, with the sole exception that positions
$4$ and $0$ may both carry $1$. Every nonnegative integer has a unique
representation after leading zeros are removed. A legal word supported on
positions $0,\ldots,m$ has value at most $U_{m+1}-1$.
\end{lemma}
\begin{proof}
For $p\ge5$, the recurrence gives
\begin{equation}\label{eq:gap}
 U_{p+1}-U_p=U_{p-4}.
\end{equation}
Together with the initial values, this proves that the weights are strictly
increasing. If the greedy algorithm selects $U_p$, its remainder is smaller
than $U_{p+1}-U_p$. For $p\ge5$ the next selected position is therefore at
most $p-5$. At $p=4$ the remainder is at most $1$, allowing position $0$;
at $p\le3$ the remainder is zero.

Conversely, induct on the largest permitted position $m$. For $m\le3$ only
singletons can occur. For $m=4$ the maximum is $U_4+U_0=6=U_5-1$.
If $m\ge5$ is selected, the lower positions are at most $m-5$, so induction
and~\eqref{eq:gap} bound the total by
\[
 U_m+(U_{m-4}-1)=U_{m+1}-1.
\]
If $m$ is not selected, use the preceding induction case. The largest
selected position is thus the first greedy choice, and the remaining
positions are greedy by induction. Existence follows because $U_0=1$;
the deterministic greedy algorithm proves uniqueness.
\end{proof}

The exception is positional: $10001$ is legal, but $100010$ is not. A
partial DFA for $G$ records the number of zeros still required before the
next $1$, together with a terminal state for the exceptional final $1$.
Every state of this partial automaton is accepting.

\subsection{An exact carry invariant}
The weights do not satisfy the homogeneous recurrence at $p=3$.
To account for this, put $W_0=2,W_1=2,W_2=3$ and
$W_p=W_{p-2}+W_{p-3}$ for $p\ge3$. Then
$W_p=U_p$ for $p\ge1$ and $W_0=U_0+1$.

\begin{lemma}\label[lemma]{lem:carry}
Read a signed digit word from most to least significant, starting with
$r=(0,0,0)$ and updating on digit $z$ by
\begin{equation}\label{eq:carry}
 T(r,z)=(r_2+z,r_0+r_2,r_1),\qquad
 \Lambda(r)=2r_0+2r_1+3r_2.
\end{equation}
After the final digit $z_0$, its represented value is
\begin{equation}\label{eq:carryvalue}
 \sum_p z_pU_p=\Lambda(r)-z_0.
\end{equation}
The empty word has both sides zero.
\end{lemma}
\begin{proof}
The update multiplies the digit polynomial by $X$, appends $z$, and reduces
modulo $X^3-X-1$. The linear map $X^p\mapsto W_p$ vanishes on multiples
of that polynomial, by the homogeneous recurrence. Its value on the
remainder is $2r_0+2r_1+3r_2$. Replacing $W$ by $U$ subtracts exactly $z_0$.
\end{proof}

For addition on digit tracks $x,y,z$, feed the signed digits $x+y-z$ to
\eqref{eq:carry}. Retain only carries in $[-8,8]^3$, remember the last
signed digit, impose membership in $G$ on all three tracks, and accept when
$\Lambda(r)$ equals that last digit. Trim and minimize the resulting
partial DFA. It has $411$ states; throughout this paper state counts exclude
the implicit rejecting sink.

\begin{proposition}[Certified addition]\label[proposition]{prop:addition}
Every triple accepted by the resulting relation $\Add$ satisfies $x+y=z$
on represented values. Moreover, for every equally long pair of words in
$G$, prefixing three zero columns gives an input pair with an accepted output.
\end{proposition}
\begin{proof}
Soundness follows immediately from \cref{lem:carry} and is independent of
the carry cutoff. Completeness is the following finite-language inclusion,
checked exhaustively by projection, determinization and product reachability:
\begin{equation}\label{eq:addtotal}
 000(G\times G)\ \subseteq\ \exists z\,\Add(x,y,z).
\end{equation}
Here $000$ means three zero columns on both input tracks; the quantified
output has the same total length. Every pair of integers has an equally
padded input representation of this form. An accepted output exists by
\eqref{eq:addtotal}, and its value is the sum by soundness.
\end{proof}

The bound on the carries only restricts the candidate automaton;
completeness comes from the inclusion~\eqref{eq:addtotal}.

\section{Identification with the recurrence}\label{sec:identification}
Define
\[
 b(n)=\sh(n)+E(\rep(n)).
\]
State $0$ of $E$ loops on digit $0$, so leading zeros do not change $b$.
To construct its synchronized graph, use a two-digit delay on the input
track. It produces the digits of the shifted word, padded with two leading
zeros. Apply the carry update to the difference between these delayed digits
and the output digits, and run $E$ on the input. The acceptance condition is
\[
 \Lambda(r)-z_0+E(q)=0,
\]
with both tracks in $G$. Every accepted pair $(n,v)$ then satisfies $v=b(n)$.
Using the same finite carry cutoff, trimming and minimization give a
$72$-state relation $\Graph$.

\begin{proposition}\label[proposition]{prop:graph}
The graph relation is total on $000G$, and its represented function satisfies
\begin{equation}\label{eq:bbounds}
 b(0)=0,\qquad b(1)=1,\qquad 1\le b(n)\le n\quad(n\ge1).
\end{equation}
\end{proposition}
\begin{proof}
Projection and inclusion verify
$000G\subseteq\exists v\,\Graph(n,v)$. Soundness makes every such output
equal to the integer formula defining $b$, so in particular $b$ is
nonnegative. The language of inputs satisfying
$\exists v,t\,[\Graph(n,v)\wedge\Add(v,t,n)]$ contains $000G$.
The language satisfying $\Graph(n,0)$ contains no representation of a
positive integer. These two finite checks give the inequalities.
The initial values are read directly from $E$.
\end{proof}

Intersections implement conjunctions of synchronized relations; existential
projection removes tracks. To verify the recurrence, define $\Good(n)$ by
\begin{equation}\label{eq:good}
\begin{split}
\exists m,p,h,v,k,w,u\;[\;&m+1=n\ \wedge\ \Graph(m,p)
 \ \wedge\ h+p=n\\
 &\wedge\ \Graph(h,v)\ \wedge\ k+v=n
 \ \wedge\ \Graph(k,w)\\
 &\wedge\ \Graph(n,u)\ \wedge\ u+w=n\;].
\end{split}
\end{equation}
Every addition here is implemented with $\Add$. Some intermediate relations
are listed in \cref{tab:relations}.

\begin{table}[tb]
\centering\small
\begin{tabular}{lr}
\toprule
Represented relation & States\\
\midrule
$x+y=z$ & 411\\
$v=b(n)$ & 72\\
$p=b(n-1)$, $n\ge1$ & 79\\
$v=b(n-b(n-1))$ & 144\\
$w=b(n-b(n-b(n-1)))$ & 80\\
$\Good(n)$ & 7\\
\bottomrule
\end{tabular}
\caption{Selected partial DFAs used in the recurrence verification.}
\label{tab:relations}
\end{table}

\begin{proposition}\label[proposition]{prop:recurrence}
For every $n\ge2$,
\begin{equation}\label{eq:brec}
 b(n)+b\bigl(n-b\bigl(n-b(n-1)\bigr)\bigr)=n.
\end{equation}
\end{proposition}
\begin{proof}
The exact inclusion check verifies that $000G$, restricted to represented
values at least $2$, is contained in $\Good$. Soundness of each component
of~\eqref{eq:good} then gives~\eqref{eq:brec} for every such integer.
The witnesses are natural numbers. By~\eqref{eq:bbounds}, their values
are at most $n$, so their canonical representations fit within that of $n$;
the three additional leading zeros provide the input padding used in the
component checks, so the projections lose no witness.
\end{proof}

\begin{proof}[Proof of the identification]
For $n\ge2$, \eqref{eq:bbounds} gives
\[
 1\le n-b(n-1)<n,\qquad
 1\le n-b\bigl(n-b(n-1)\bigr)<n.
\]
The first inequality uses $1\le b(n-1)\le n-1$; the second applies the
same bounds to the first inner index. Starting with $a(0)=b(0)$ and
$a(1)=b(1)$, strong induction shows that~\eqref{eq:recurrence} is
well-defined and agrees with $b$ at every index. This proves existence,
uniqueness, nonnegativity, and~\eqref{eq:identification}.
\end{proof}

For the difference identity, synchronize the graph at $n$ and $n+1$ with
the output of $D$ and addition. The resulting $12$-state input relation
expresses $b(n)+D(\rep(n))=b(n+1)$; exact inclusion verifies that it contains
$000G$. This proves~\eqref{eq:difference}, including $D(\eps)=1$ at $n=0$.

\section{Discrepancy and effective extrema}\label{sec:discrepancy}
Set $f(n)=a(n)-cn$. Define
\[
 \delta_0=-c,\qquad \delta_1=-2c,\qquad
 \delta_p=U_{p-2}-cU_p\quad(p\ge2).
\]
By~\eqref{eq:identification}, for selected positions $J$,
\begin{equation}\label{eq:discrepancy}
 f(n)=E(\rep(n))+\sum_{p\in J}\delta_p.
\end{equation}

\subsection{A contracting quadratic form}
\begin{lemma}\label[lemma]{lem:contraction}
For $p\ge3$,
\begin{equation}\label{eq:decay}
 |\delta_p|<\frac5{16}\left(\frac78\right)^{p-3}.
\end{equation}
\end{lemma}
\begin{proof}
For $i\ge1$, let
$F_i=U_{i+2}+\rho U_{i+1}+\rho^{-1}U_i$.
The weight recurrence and $\rho^2=1+\rho^{-1}$ give $F_{i+1}=\rho F_i$.
Consequently
\begin{equation}\label{eq:deltarec}
 \delta_{p+2}+\rho\delta_{p+1}+\rho^{-1}\delta_p=0\quad(p\ge3).
\end{equation}
The cutoff $p\ge3$ ensures that both homogeneous identities used here
avoid the anomalous initial weight.

Consider the positive quadratic form
\begin{align*}
 Q(x,y)&=\rho^{-1}x^2+\rho xy+y^2\\
 &=\left(y+\frac{\rho x}{2}\right)^2+
       \frac{3-\rho}{4\rho}x^2.
\end{align*}
Direct expansion gives
\begin{equation}\label{eq:energy}
 Q(y,-\rho y-\rho^{-1}x)=\rho^{-1}Q(x,y).
\end{equation}
All constants can be bounded rationally. The signs of $t^3-t-1$ at
$331/250$ and $53/40$ give
$331/250<\rho<53/40$. Hence $\rho>64/49$, $\rho^{-1}<19/25$,
$\rho>33/25$, and $(3-\rho)/(4\rho)>5/16$.
Using $\delta_3=2-4c$ and $\delta_4=3-5c$, these brackets imply
\[
 -\frac{141}{500}<\delta_3<-\frac{139}{500},\qquad
 \frac{147}{1000}<\delta_4<\frac{153}{1000}.
\]
Keeping the sign of the mixed term,
\begin{align*}
 Q(\delta_3,\delta_4)
 &<\frac{19}{25}\left(\frac{141}{500}\right)^2
   +\left(\frac{153}{1000}\right)^2
   -\frac{33}{25}\frac{139}{500}\frac{147}{1000}\\
 &=\frac{747603}{25000000}<\frac3{100}.
\end{align*}
Iterating~\eqref{eq:energy} along~\eqref{eq:deltarec} gives
\[
 |\delta_p|^2<\frac{12}{125}
       \left(\frac{49}{64}\right)^{p-3}
 <\frac{25}{256}\left(\frac{49}{64}\right)^{p-3},
\]
which proves~\eqref{eq:decay}.
\end{proof}

For $K\ge5$, selected positions at least $K$ are spaced by at least five.
The absolute contribution of this tail is therefore strictly smaller than
\begin{equation}\label{eq:tail}
 R_K=\frac{(5/16)(7/8)^{K-3}}{1-(7/8)^5}.
\end{equation}
The finitely many lower positions and the bounded output $E$ already imply
that $f$ is bounded. To obtain useful constants we optimize the lower digits
jointly with the automaton state.

\subsection{Suffix optimization and realizable witnesses}
Following the finite-tail method of~\cite{Letouzey}, we retain the correction
state and reaching prefixes so that extremal paths represent actual integers.
Let $\alpha<c<\beta$ be a rational bracket. After an arbitrary high prefix,
there are $28$ reachable pairs consisting of an $E$-state and a $G$-state.
For each such pair fix a finite reaching prefix. A breadth-first search
finds prefixes of length at most $H=13$.

For each legal path of length $K$ from these states, its contribution from
the last $K$ digits, including the final output of $E$, has the form
$A-cB$, where $A\in\Z$ and $0\le B<U_K$. Dynamic programming on the
finite state set computes the exact rational extrema
\begin{equation}\label{eq:suffixbounds}
 L_K=\min(A-\beta B),\qquad V_K=\max(A-\alpha B),
\end{equation}
retaining one path attaining each endpoint. By~\eqref{eq:tail},
\begin{equation}\label{eq:globalbound}
 L_K-R_K<f(n)<V_K+R_K\quad(n\ge0).
\end{equation}
Short representations are padded with zeros; this changes neither legality
nor the output of $E$.

\begin{proposition}[Effective global extrema]\label[proposition]{prop:extrema}
Let $m=\inf_{n\ge0} f(n)$ and $M=\sup_{n\ge0} f(n)$.
The suffix construction provides rational enclosures of both numbers with
width at most
\begin{equation}\label{eq:convergencerate}
 2R_K+(\beta-\alpha)(U_K+U_{K+H}),\qquad H=13.
\end{equation}
Choosing $\beta-\alpha=2^{-4K-64}$ gives widths tending to zero.
\end{proposition}
\begin{proof}
Attach the recorded prefixes to extremizing suffixes and bound the omitted
tail by $R_K$. The full argument and witness enclosures are given in
Appendix~S1.
\end{proof}

\subsection{Numerical certificates and the corrected offsets}
Exact rational optimization at cutoff $K=200$ gives the outward-rounded
enclosures below; the cutoff table and witnesses are in Appendix~S1.
The weaker six-decimal bound~\eqref{eq:mainbound} is also checked in Lean
using a separate $K=105$ certificate. The refined enclosures are
\begin{align}
 -1.049233035332&\le m\le-1.049233035329,\label{eq:minnum}\\
  1.155080224901&\le M\le 1.155080224905.\label{eq:maxnum}
\end{align}
We do not know whether these extrema are attained.

Put $k=a(n)-\lfloor cn\rfloor$. Since
$k=f(n)+\{cn\}$, \eqref{eq:mainbound} gives $-2<k<3$.
The integer $k$ belongs to $\{-1,0,1,2\}$. All four values occur:
\begin{center}
\begin{tabular}{rrrr}
\toprule
$n$ & $a(n)$ & $\lfloor cn\rfloor$ & Offset\\
\midrule
1167 & 664 & 665 & $-1$\\
2 & 1 & 1 & $0$\\
1 & 1 & 0 & $1$\\
93 & 54 & 52 & $2$\\
\bottomrule
\end{tabular}
\end{center}
The floor values are certified by the signs at consecutive integers of
$k^3-k^2n+2kn^2-n^3=n^3P(k/n)$. The sequence values at these small
indices can be computed directly from~\eqref{eq:recurrence}.
Dividing~\eqref{eq:mainbound} by $n$ proves $a(n)/n\to c$.

\section{A morphic description of the difference word}\label{sec:morphic}
Let $d(n)=a(n+1)-a(n)$ for $n\ge0$. Its first symbol is $d(0)=1$;
thus the positive-index difference word is obtained by deleting this initial
symbol. Define the canonical language
\[
 L=\{\eps\}\cup\bigl(G\cap1\{0,1\}^*\bigr),\qquad 0<1.
\]

\begin{lemma}\label[lemma]{lem:order}
Genealogical order on $L$, namely length followed by lexicographic order,
coincides with increasing represented value.
\end{lemma}
\begin{proof}
A canonical word of length $p+1$ has value at least $U_p$, whereas a shorter
word has value at most $U_p-1$ by \cref{lem:greedy}. For equal lengths,
consider the largest position $p$ where two words differ. The lexicographically
larger word contributes $U_p$ there; the entire lower suffix of the smaller
word contributes at most $U_p-1$. Their higher parts agree, proving strict
value order. Greedy existence and uniqueness complete the claim.
\end{proof}

The language $L$ is infinite and regular and contains $\eps$. Therefore
$(L,\{0,1\},<)$ is an abstract numeration system enumerating precisely
$\rep(0),\rep(1),\ldots$. The verified $D$-DFAO makes $d$ automatic in
this system. The Rigo--Maes theorem~\cite{RM}, explicitly restated in the
one-dimensional case of~\cite[Theorem~24]{CKR}, implies that $d$ is morphic.
The canonical language $L$ gives each integer a unique representation,
as required for this enumeration.

Our Lean proof of the concrete presentation uses the following direct
argument. For each fixed
length $m$, the legal padded words enumerate exactly the integers below $U_m$.
Their lexicographic rank equals their represented value: whenever a digit
$1$ is legal, the preceding digit-$0$ subtree has size $U_j$, where $j$ is
the remaining suffix length. Induction on the input word proves that the
corresponding iterate of the state morphism has the correct $D$-state at
every index. Fixing the length gives each integer in this range a unique
padded representation.

\subsection{A non-erasing presentation}
We give an explicit presentation. Form the product of $D$ with the partial
automaton for $G$, starting after the first digit $1$. Exactly $27$ pairs
are reachable, with distinct $D$-states $q_1,\ldots,q_{27}$.
Only $q_6$ and $q_{21}$ have no legal children. Every other state has a
nonterminal child on digit $0$. These statements are checked by exhaustive
traversal of the finite product; the script also verifies that $D$ is
defined on every legal continuation.

Introduce a marker $S$ with output $1$, corresponding to $\eps$. Define a
morphism $\mu$ by
\[
 \mu(S)=S q_1,\qquad
 \mu(q)=\text{the legal children of $q$ in digit order $0,1$}.
\]
In particular $\mu(q_6)=\mu(q_{21})=\eps$. Its fixed point
\[
 x=\mu^\omega(S)=S q_1\mu(q_1)\mu^2(q_1)\cdots
\]
enumerates the representation tree level by level. Every level contains
$10\cdots0$, so the word is infinite. Let $\tau$ output $1$ on $S$ and
the $D$-output on each $q_i$. By \cref{lem:order} and
\eqref{eq:difference}, $\tau(x)=d(0)d(1)\cdots$.

\begin{proposition}\label[proposition]{prop:morphism}
Let $I=\{S,q_1,\ldots,q_{27}\}\setminus\{q_6,q_{21}\}$. There are
non-erasing morphisms $\nu:I^*\to I^*$ and $\eta:I^*\to\{0,1\}^*$,
given in \cref{tab:morphism}, such that
\begin{equation}\label{eq:morphic}
 d(0)d(1)d(2)\cdots=\eta\bigl(\nu^\omega(S)\bigr).
\end{equation}
\end{proposition}
\begin{proof}
Let $\pi$ delete $q_6,q_{21}$ and fix every other letter. On $I$, put
$\nu=\pi\mu$, $h=\mu$, and $\eta=\tau h$.
Since every nonterminal state has a nonterminal child, both $\nu$ and
$\eta$ are non-erasing. The relation $\pi\mu=\nu\pi$ gives
$y=\pi(x)=\nu^\omega(S)$. Because $\mu$ annihilates the deleted letters,
\[
 x=\mu(x)=\mu(\pi(x))=h(y).
\]
Applying $\tau$ proves~\eqref{eq:morphic}.
\end{proof}

The output morphism $\eta$ has images of length one or two; it is not a
letter coding on this same alphabet. We do not know whether this presentation
is minimal, nor how it relates to the six-letter substitution found
experimentally by the first author.

\begin{table}[tb]
\centering\small
\input{generated/morphism-table.tex}
\caption{The non-erasing substitution and output morphism in
\cref{prop:morphism}. The initial letter is $S$.}
\label{tab:morphism}
\end{table}

\begin{corollary}\label[corollary]{cor:word}
The frequency of $1$ in $d$ is $c$, the word is not ultimately periodic,
and its least balance constant is $4$.
\end{corollary}
\begin{proof}
There are $a(N)$ ones in the first $N$ symbols, so their frequency tends
to $c$. The rational-root test shows that $P$ has no rational root;
an ultimately periodic binary word would have rational letter frequency.
For a length-$\ell$ factor beginning at $i$, the number of ones is
\[
 a(i+\ell)-a(i)=c\ell+f(i+\ell)-f(i).
\]
The bounds in~\eqref{eq:mainbound} have width $2.204315$. Counts of ones
in two equal-length factors thus differ by less than $4.408630$, hence by
at most $4$ since the difference is integral. The same holds for zeros.
For sharpness, direct evaluation gives
$a(135)=76$, $a(207)=119$, $a(214)=123$, and $a(286)=162$.
The length-$72$ factors starting at $135$ and $214$ therefore contain
$43$ and $39$ ones, respectively. Their difference is $4$, ruling out
every smaller balance constant. These evaluations are also kernel checked.
\end{proof}

\section{Verification and further questions}\label{sec:verification}
The finite-state assertions are proved by exhaustive language inclusion,
using integer transition tables. Two separately written Python
implementations reproduce the recurrence checks, and Walnut independently
reconstructs six quantified relations. The accompanying Lean development
checks the entire identification chain, the six-decimal global bound,
the exact offset set, the concrete morphic identity, the frequency,
nonperiodicity, and the least balance constant $4$. Its final axiom closures
contain only \code{propext}, \code{Classical.choice}, and \code{Quot.sound}.

Lean also checks a computable extrema algorithm: exhaustive search over
$n<U_{K+16}$ with rational root brackets and a proved geometric error.
The faster suffix optimizer used for
\eqref{eq:minnum}--\eqref{eq:maxnum} is a separate exact computation;
those narrower decimal enclosures are not Lean theorems. Appendix~S2
records verification scope and reproducible build instructions.

Open questions include the relation with the original six-letter
substitution, the minimality of the morphic alphabet, the attainment of the
global extrema and their relation to asymptotic extrema, and the dynamical
description of the difference word, for instance as a coding of a toral
rotation.

\section*{Contributions and use of AI systems}
\begingroup\small
B.C. identified the Padovan description of the sequence as a two-position
shift with a finite-state correction, supplied the correction automata, and
wrote an independent verification of the automata proof. W.Z. and H.M. found
the counterexample to the original offset conjecture, built the certified
adder, the identification proof and the discrepancy bounds, and developed
the Lean and Walnut verifications.

The first author used Claude (Anthropic)
as an experimental laboratory for computations, including verification code,
and for cross-checks against the literature. H.M. and W.Z. used OpenAI models,
primarily GPT-5.6 and GPT-6 Astra, through Codex to assist with counterexample
searches, exploration of proof strategies, automata construction, and the
writing, debugging, and review of Python code, Lean proofs, and Walnut
specifications. These tools also assisted with cross-checking implementations
and computational results, comparing the arguments with the literature,
and drafting and revising the manuscript.

Model-generated suggestions and code were subjected to the formal and
computational checks described in this paper. The Lean kernel checked the
formalized theorem statements and proofs; Appendix~S2 records their scope
and axiom closures.
All authors retain responsibility for the mathematical claims, verification
code, and final text.
\endgroup

\appendix
\renewcommand{\thesection}{S\arabic{section}}
\renewcommand{\thetable}{S\arabic{table}}
\renewcommand{\theHtable}{appendix.S\arabic{table}}
\setcounter{table}{0}
\numberwithin{equation}{section}
\input{appendix-body.tex}

\end{document}

%% file: generated/morphism-table.tex
\begin{tabular}{lll@{\qquad\qquad}lll}
\toprule
$q$ & $\nu(q)$ & $\eta(q)$ & $q$ & $\nu(q)$ & $\eta(q)$\\
\midrule
$S$ & $S\,q_{1}$ & \texttt{10} & $q_{14}$ & $q_{17}$ & \texttt{0}\\
$q_{1}$ & $q_{2}$ & \texttt{1} & $q_{15}$ & $q_{18}$ & \texttt{1}\\
$q_{2}$ & $q_{3}$ & \texttt{0} & $q_{16}$ & $q_{19}$ & \texttt{0}\\
$q_{3}$ & $q_{4}$ & \texttt{1} & $q_{17}$ & $q_{20}$ & \texttt{01}\\
$q_{4}$ & $q_{5}$ & \texttt{10} & $q_{18}$ & $q_{22}$ & \texttt{0}\\
$q_{5}$ & $q_{7}\,q_{8}$ & \texttt{01} & $q_{19}$ & $q_{23}$ & \texttt{11}\\
$q_{7}$ & $q_{9}\,q_{10}$ & \texttt{11} & $q_{20}$ & $q_{24}\,q_{1}$ & \texttt{10}\\
$q_{8}$ & $q_{11}$ & \texttt{0} & $q_{22}$ & $q_{9}$ & \texttt{11}\\
$q_{9}$ & $q_{5}\,q_{12}$ & \texttt{10} & $q_{23}$ & $q_{25}\,q_{10}$ & \texttt{01}\\
$q_{10}$ & $q_{13}$ & \texttt{0} & $q_{24}$ & $q_{26}\,q_{1}$ & \texttt{00}\\
$q_{11}$ & $q_{14}$ & \texttt{1} & $q_{25}$ & $q_{20}\,q_{27}$ & \texttt{01}\\
$q_{12}$ & $q_{15}$ & \texttt{1} & $q_{26}$ & $q_{23}\,q_{8}$ & \texttt{11}\\
$q_{13}$ & $q_{16}$ & \texttt{1} & $q_{27}$ & $q_{15}$ & \texttt{1}\\
\bottomrule
\end{tabular}

%% file: appendix-body.tex
\section{Effective extrema: proof and numerical certificates}
Write $f(n)=a(n)-cn$, $m=\inf_{n\ge0}f(n)$, and
$M=\sup_{n\ge0}f(n)$. For $K\ge5$, the sparse-tail bound is
\[
 R_K=\frac{(5/16)(7/8)^{K-3}}{1-(7/8)^5}.
\]
Let $\alpha<c<\beta$ be a rational bracket. A legal suffix of length $K$,
including the correction output, contributes $A-cB$, where
$A\in\mathbb Z$ and $0\le B<U_K$. Optimize over suffix paths from the
$28$ reachable correction/greedy state pairs:
\[
 L_K=\min(A-\beta B),\qquad V_K=\max(A-\alpha B).
\]
Each pair has a recorded reaching prefix of length at most $H=13$.
The global bound is $L_K-R_K<f(n)<V_K+R_K$.

\subsection{Proof of the effective-extrema proposition}
Join a minimizing suffix path to its recorded reaching prefix, producing
a legal word for an integer $n_-$. Likewise obtain $n_+$ from a maximizing
path. Each concatenation is legal because its prefix reaches the initial
state of the chosen suffix. Since their lengths are at most $K+H$, both integers
are smaller than $U_{K+H}$ by the greedy-language bound in the main text.

Use the representation formula and root bracket to enclose their
discrepancies. If $u_-$ is the upper endpoint for $f(n_-)$ and $\ell_+$
is the lower endpoint for $f(n_+)$, then
\begin{equation}
 L_K-R_K\le m\le u_-,\qquad
 \ell_+\le M\le V_K+R_K.
\end{equation}
The chosen minimizing suffix has true value at most
$L_K+(\beta-\alpha)U_K$. Its high-prefix contribution is less than $R_K$
in absolute value, and enclosing its full discrepancy adds at most
$(\beta-\alpha)U_{K+H}$. Thus each enclosure has width at most
\begin{equation}
 2R_K+(\beta-\alpha)(U_K+U_{K+H}).
\end{equation}
The maximum is symmetric. Exact bisection of
$P(t)=t^3-t^2+2t-1$ on $[0,1]$ supplies a bracket of width
$2^{-4K-64}$. Since $U_j<2^{j+1}$, the second term tends to zero;
$R_K$ tends to zero geometrically as well. This proves the proposition.

\subsection{Exact computation and witnesses}
Table~\ref{tab:bounds} gives rational enclosures of the global extrema
with outward-rounded decimal endpoints.
\begin{table}[htbp]
\centering\small
\input{generated/bounds-table.tex}
\caption{Global discrepancy bounds. The last column bounds the larger
exact extrema-enclosure width, rounded upwards.}
\label{tab:bounds}
\end{table}

For $K=200$, the root bracket uses $864$ exact bisection steps.
Representative witnesses are
\begin{align*}
 n_-&=27957586325095713833899543,&
 a(n_-)&=15931359127095736187627998,\\
 n_+&=9231177514959072345795515,&
 a(n_+)&=5260296881378963990671881.
\end{align*}
The representation formula evaluates these large values directly from
their greedy representations. Exact discrepancy brackets
and greedy words are recorded in \path{develop/automata-results/extrema.json}.
For suffix lengths $5,8,12$, exhaustive enumeration agrees with the
dynamic program. All eight stored witnesses at four cutoffs are accepted
by the synchronized graph, and the two small witnesses also agree with
direct recurrence evaluation.

\section{Verification scope and reproduction}
\subsection{Finite-state and independent checks}
Products implement intersection, subset construction implements
existential projection, and inclusion is decided by exhaustive reachability
against the complement of the target. Undefined transitions reject.
Each inclusion either returns a counterexample or exhausts the finite
search space.

Two separately written Python implementations reproduce the principal
recurrence checks. The second imports the literal correction table,
but none of the first implementation's automata or transformation code.
At carry cutoffs $8$ and $10$ it passes adder totality, graph totality,
the recurrence and value bounds. The morphic presentation and representation
order were also compared with the literal recurrence through $100{,}000$
terms; this is a diagnostic, not an all-index proof.

Walnut~7.1.0 imports the base adder, candidate graph, greedy domains,
constants and difference relation, then reconstructs all intermediate
quantified relations. Adder totality, graph totality, value bounds,
positivity, the recurrence and the difference identity all return
\code{TRUE}. Integer interpretation still rests on the carry proof.

\subsection{Lean scope}
The Lean development checks canonical greedy representations, recognition,
carry and delay semantics, finite-language inclusion and projection,
totality, and identification with the literal nested recurrence. It also
proves the exact four offsets, slope limit, concrete $26$-letter morphic
identity, irrational frequency, nonperiodicity, and exact balance constant
$4$. In \path{develop/lean/BalanceSharpness.lean}, the two length-$72$
factors have $43$ and $39$ ones. The theorem \code{exact\_balance} states
that a natural number $C$ is a uniform balance bound if and only if $4\le C$.

A real quadratic contraction and sparse tail supply the analytic bound.
The $K=105$ Bellman certificate with a $160$-bit dyadic root bracket
proves $-1.049234<f(n)<1.155081$ in the kernel.
The formal extrema algorithm searches $n<U_{K+16}$ and uses a bracket of
width $2^{-4K-64}$. A prefix-compression lemma bounds the discrepancy
outside this window, proving rational enclosures and convergence.
This exhaustive algorithm establishes computability; it is distinct from
the faster suffix optimizer used for the narrow $K=200$ decimals.

Final theorem axiom closures contain only \code{propext},
\code{Classical.choice}, and \code{Quot.sound}. No admitted proofs,
extra axioms, or native decision axioms are used. Kernel reduction checks
the generated finite certificates. The narrow $K=200$ decimals and
the original six-letter identity are not asserted as Lean theorems.

\newpage
\subsection{Reproduction commands}
The reproducibility package is available at
\url{https://github.com/the-omega-institute/a076502-padovan}.
Version v1.0.1 is archived at
\url{https://doi.org/10.5281/zenodo.22979217}.

From the package's \path{develop/} directory, standard-library Python~3
reconstructs the finite-state proof and certificates:
\begin{verbatim}
python3 automata_proof.py --bound 8
python3 discrepancy_certificate.py
python3 morphic_certificate.py
python3 extrema_certificate.py
python3 paper_checks.py
\end{verbatim}
The original literal tables and independent verifier are retained,
unmodified, in \path{correspondence/attachments/}.

Reports in \path{develop/automata-results/} include
\path{report.json}, \path{morphic-presentation.json}, \path{extrema.json},
\path{paper-checks.json}, and \path{walnut.json}.
The optional million-term diagnostic is \code{independent\_checks.py}.
Walnut can be replayed using \code{walnut\_check.py} with Java and the
pinned jar described in \path{develop/walnut/README.md}.

The Lean toolchain is $4.33.0$, with Mathlib pinned to
\begin{quote}\small\ttfamily
db584cd6d46c92f209a44c0f1c829460d327499d.
\end{quote}
With Elan, Git, Python~3 and network access available, run from the
extracted research package root:
\begin{verbatim}
python3 develop/clean_ci.py
\end{verbatim}
This command rebuilds the Lean development in a fresh Lake project
with the pinned dependencies and records source hashes and final axiom
closures in \path{develop/ci-results/}. The README of the package gives
the details.

\clearpage
\section{The finite-state tables}
Table~\ref{tab:dfao} specifies both input transition functions and outputs.
The initial state is $0$ for both automata. A dash denotes an undefined
transition. The transition functions are not identical: state $22$ has
different digit-$1$ successors in $E$ and $D$. Representations are always
restricted to the greedy language.
\begin{table}[htbp]
\centering\small
\input{generated/dfao-table.tex}
\caption{The $E$- and $D$-DFAOs used in the main text. Columns $0$ and $1$
give successor states; ``out'' gives the output.}
\label{tab:dfao}
\end{table}

%% file: generated/bounds-table.tex
\begin{tabular}{rrrr}
\toprule
$K$ & Lower endpoint & Upper endpoint & Extrema width\\
\midrule
20 & -1.086486290247 & 1.198282591609 & 0.079247100790\\
70 & -1.049288300721 & 1.155139949837 & 0.000096687582\\
120 & -1.049233123804 & 1.155080305728 & 0.000000115468\\
200 & -1.049233035332 & 1.155080224905 & 0.000000000003\\
\bottomrule
\end{tabular}

%% file: generated/dfao-table.tex
\begin{tabular}{r rrr rrr}
\toprule
 & \multicolumn{3}{c}{$E$} & \multicolumn{3}{c}{$D$}\\
\cmidrule(lr){2-4}\cmidrule(lr){5-7}
State & $0$ & $1$ & out & $0$ & $1$ & out\\
\midrule
0 & 0 & 1 & 0 & 0 & 1 & 1\\
1 & 2 & -- & 1 & 2 & -- & 0\\
2 & 3 & -- & 1 & 3 & -- & 1\\
3 & 4 & -- & 1 & 4 & -- & 0\\
4 & 5 & 6 & 0 & 5 & 6 & 1\\
5 & 7 & 8 & 0 & 7 & 8 & 1\\
6 & -- & -- & 1 & -- & -- & 0\\
7 & 9 & 10 & 0 & 9 & 10 & 0\\
8 & 11 & -- & 0 & 11 & -- & 1\\
9 & 5 & 12 & 0 & 5 & 12 & 1\\
10 & 13 & -- & 1 & 13 & -- & 1\\
11 & 14 & -- & 2 & 14 & -- & 0\\
12 & 15 & -- & 1 & 15 & -- & 0\\
13 & 16 & -- & 1 & 16 & -- & 0\\
14 & 17 & -- & 0 & 17 & -- & 1\\
15 & 18 & -- & 1 & 18 & -- & 1\\
16 & 19 & -- & 1 & 19 & -- & 1\\
17 & 20 & 21 & 1 & 20 & 21 & 0\\
18 & 22 & -- & 1 & 22 & -- & 1\\
19 & 23 & 21 & 1 & 23 & 21 & 0\\
20 & 24 & 1 & 0 & 24 & 1 & 0\\
21 & -- & -- & 0 & -- & -- & 1\\
22 & 9 & 6 & 0 & 9 & 21 & 0\\
23 & 25 & 10 & -1 & 25 & 10 & 1\\
24 & 26 & 1 & 0 & 26 & 1 & 1\\
25 & 20 & 27 & 1 & 20 & 27 & 0\\
26 & 23 & 8 & 1 & 23 & 8 & 0\\
27 & 15 & -- & 0 & 15 & -- & 1\\
\bottomrule
\end{tabular}